\documentclass{amsart}
\usepackage{amssymb}

\usepackage{amsfonts}
\usepackage{color}
\usepackage{graphicx}
\usepackage[bookmarksnumbered,colorlinks, linkcolor=blue, citecolor=red, pagebackref, bookmarks, breaklinks]{hyperref}

\newtheorem{thm}{Theorem}[section]
\newtheorem{cor}[thm]{Corollary}
\newtheorem{lema}[thm]{Lemma}
\newtheorem{prop}[thm]{Proposition}
\theoremstyle{definition}

\theoremstyle{remark}

\newtheorem{rem}[thm]{Remark}
\numberwithin{equation}{section}
\newcommand{\R}{\mathbb R}
\def\C{\mathbf {C}}
\def\c{\mathbf {c}}

\newcommand{\ve}{\varepsilon}
\newcommand{\lam}{\lambda}

\title{Hardy inequalities in fractional Orlicz-Sobolev spaces}
\author[A. M. Salort]
{Ariel M.  Salort}%

\address{Departamento de Matem\'atica, FCEyN - Universidad de Buenos Aires and
\hfill\break \indent IMAS - CONICET -  Ciudad Universitaria, Pabell\'on I (1428) Av. Cantilo s/n. \hfill\break \indent Buenos Aires, Argentina.}

\email[A.M. Salort]{asalort@dm.uba.ar}
\urladdr{http://mate.dm.uba.ar/~asalort}

\begin{document}



\begin{abstract}
In this article we prove both norm and modular Hardy inequalities for a class functions in one-dimensional fractional Orlicz-Sobolev spaces. 
\end{abstract}

\vspace*{-1cm}
\maketitle
\vspace{-1cm}
\setlength{\parskip}{0.06em}
\setlength{\parskip}{0.5em}

\section{Introduction}

In the early 1920's,  in the seminal article \cite{Ha},  G.H. Hardy  obtained inequalities of the form
\begin{equation} \label{hardy.intro}
\int_0^\infty \frac{|u(x)|^p}{x^{p}} \,dx \leq \left(\frac{p}{p-1} \right)^p \int_0^\infty |u'(x)|^p\,dx
\end{equation}
where $p>1$ and $u$ is a nonnegative measurable function defined on $(0,\infty)$.

Throughout  the decades that followed, many authors contributed to characterize the family of weights $v,w$ and powers $p,q$ for which inequalities of the type
$$
\left(\int_0^\infty |u(x)|^p v(x)\,dx\right)^\frac1p \leq \C_H \left(\int_0^\infty |u'(x)|^q w(x)\,dx\right)^\frac1q
$$
are valid for a suitable   positive constant $\C_H$ independent of $u$. See  for instance   the pioneering works \cite{Brandly, Talenti,Tom}. We also refer  to the   seminal books \cite{KP, OK}.

Orlicz spaces play a fundamental role when describing  phenomena with non-standard growth. See  \cite{KR,KJF,RR}. Generalizations of Hardy type inequalities to the Orlicz space structure provide  for a family of inequalities admitting behaviors more general than powers. In the last years many authors leaded the task of characterizing  the class of admissible  weights $v,w,a,b$ and Young functions $P,Q$ for which is valid an inequality as follows
$$
Q^{-1}\left(\int_0^\infty Q(a(x)|Tu(x)|) v(x)\,dx \right) \leq \C_H P^{-1} \left( \int_0^\infty P(  b(x)|u(x)|)w(x)\,dx \right),
$$
where $Tu(x)=\int_0^x \mathcal{K}(x,y)u(y)\,dy$ is the generalized Hardy operator and $\C_H$ is a positive constant independent of $u$. For more details we refer  to \cite{BK1,BK2,HL,Lai}  and  references therein. In contrast with the $L^p$ case, Hardy inequalities in integral form may differ from norm inequalities. See \cite{BK1,Cia}.

Nonlocal Hardy inequalities were object of study in the last years. The nonlocal counterpart of \eqref{hardy.intro}, for suitable values of $p>1$ and $s\in(0,1)$, takes the form
$$
\int_{0}^\infty \frac{|u(x)|^p}{x^{sp}}\,dx \leq \C_H \int_0^\infty\int_0^\infty \frac{|u(x)-u(y)|^p}{|x-y|^{1+sp}}\,dxdy
$$
for functions in an appropriated  Sobolev space. Hardy inequalities for the fractional $p-$Laplacian  date back to the early 1960's and were derived independently in \cite{Gris} and \cite{J}. For the case $p=2$ see also \cite{AAS, ArSm}. See also \cite{KMP} for a different approach.

Up to   our knowledge, the only research on  Hardy inequalities in a non-local framework with  non-standard growth was  made in \cite{ACPS} for  values of $s\in(0,1)$ close to 0. See also \cite{MSV}. Therefore, the main scope of this paper is to study the validity of these  inequalities in the fractional order Orlicz-Sobolev spaces introduced in \cite{ACPS, FBS}, in both integral and norm form  for \emph{big} values of $s\in (0,1)$, that is, when $s>1/p^-$, where $p^->1$ is a fixed constant.  

In order to state our  results, recall that  a \emph{Young function} $G$ is a continuous, nonnegative, strictly increasing and convex function on $\R_+:=[0,\infty)$ (see Section \ref{sec2}), for which we assume the  growth condition
\begin{equation} \label{cond.intro} \tag{L}
1<p^-\leq \tfrac{tg(t)}{G(t)}\leq p^+<\infty \quad \text{for all }t>0
\end{equation}
for fixed constants $p^\pm$, where $g=G'$.  Given a Young function $G$, we define 
$$
L^G(\R_+):=\{u\colon\R_+\to \R \text{ measurable}\colon \Phi_G(u)<\infty\}.
$$
The \emph{modular} and the \emph{Luxemburg norm} of $u\in L^G(\R_+)$, respectively,  are defined as
$$
\Phi_{G}(u)=\int_0^\infty G(|u(x)|)\,dx \qquad \text{and}\qquad  \|u\|_G := \inf\left\{\lambda>0\colon \Phi_G\left(\frac{u}{\lambda}\right)\le 1\right\}.
$$
Given  $s\in(0,1)$,  we define the \emph{fractional Orlicz-Sobolev space}  as (see \cite{FBS})
$$
W^{s,G}(\R_+):=\left\{ u\in L^G(\R_+) \text{ such that } \Phi_{s,G}(u)<\infty \right\}.
$$
Here the \emph{fractional modular} of $u\in W^{s,G}(\R_+)$ is defined as 
$$
\Phi_{s,G}(u):=
  \int_{0}^\infty  \int_{0}^\infty G\left( \frac{|u(x)-u(y)|}{|x-y|^s} \right) \frac{ dx\,dy}{|x-y|}.
$$
These spaces are endowed with the so-called \emph{Luxemburg norm} defined as
$$
\|u\|_{s,G} := \|u\|_G + [u]_{s,G},
$$
where the  {\em $(s,G)$-Gagliardo semi-norm} reads as 
$
[u]_{s,G} :=\inf\left\{\lambda>0\colon \Phi_{s,G}\left(\frac{u}{\lambda}\right)\le 1\right\}.
$

We describe now our results. As we will see, the proof of our non-local modular inequality  is based on an astute  application of the following  modular Hardy inequality for the local  Orlicz-Sobolev space $W^{1,G}(\R_+)$ (see Lemma \ref{lema.loca} for details)
\begin{equation} \label{hardy.orlicz.local2.intro}
\int_0^\infty G\left(\left| \frac{1}{x^{s}}\int_0^x \frac{u(t)}{t}\,dt \right|\right)\,dx \leq \c_H \int_0^\infty G\left(   \frac{|u(x)|}{x^{s}}\right)\,dx.
\end{equation}
Following \cite{BK}, inequality \eqref{hardy.orlicz.local2.intro} can be obtained from the local norm  inequalities studied in \cite{Mal2,Pal}.

Finally, for simplicity in our notation,  given a Young function $G$ satisfying \eqref{cond.intro} we define the functions 
\begin{align} \label{xxx}
\psi_G(x)=
\begin{cases}
x^{p^+} &\text{if } x\geq 1\\
x^{p^-} &\text{if } x< 1\\
\end{cases} \qquad \text{and} \qquad 
\phi_G(x)=
\begin{cases}
x^{1/p^-} &\text{if } x\geq 1\\
x^{1/p^+} &\text{if } x< 1\\
\end{cases}.
\end{align}

With these preliminaries, our first result reads as follows.

\begin{thm} \label{hardy}
Let $G$ be a Young function satisfying \eqref{cond.intro} and let $s\in(0,1)$ be such that $sp^->1$. Then  for all $u\in W^{s,G}(\R_+)$ such that $
\lim_{x\to 0} \frac{1}{x}\int_0^x u(x)\,dx = u_0,
$
the following inequality holds
$$
\int_0^\infty G\left( \frac{ |u(x)-u(x_0)| }{x^s} \right)\,dx \leq \C_H \Phi_{s,G}(u), \qquad  \C_H:=\C(1+\c_H)
$$
where $\c_H$ is given in $\eqref{hardy.orlicz.local2.intro}$ and $\C:=2^{p+}$ is the doubling constant for $G$.
\end{thm}

\begin{rem}
The constant $\C_H$ in Theorem \ref{hardy} can be computed explicitly   as 
$$
\C_H=\C(1+\c_H)= 2^{p^+}\left( 1+ \psi_G\left( \frac{p^-}{sp^--1}\right) \right).
$$
\end{rem}
\begin{rem} \label{remark.111}
Observe that $u\in W^{s,G}(\R_+)$ is in fact continuous in a neighborhood $\mathcal{O}$ of the origin  when $sp^->1$ due to the embedding  of $W^{s,G}(\mathcal{O})$ into the space $C^{0,s-1/p^-}(\mathcal{O})$ of H\"older continuous functions (see Proposition \ref{continua}).
\end{rem}

From the modular inequality it is easy to deduce a norm inequality.
\begin{cor} \label{hardy.cor}
With the same assumptions of Theorem \ref{hardy},
$$
\left\| \frac{u - u_0}{x^s}\right\|_G  \leq \phi_G(\C_H) [u]_{s,G}
$$
where $\phi_G$ is given in \eqref{xxx} and $\C_H$ is the constant given in Theorem \ref{hardy}.
\end{cor}

From Theorem \ref{hardy} and Remark \ref{remark.111}  it can be easily deduced the following consequence.
\begin{cor}  \label{corol}
Let $G$ be a Young function satisfying \eqref{cond.intro} and let  $s\in (0,1)$ be such that $sp^->1$. Then  for all $u\in W^{s,G}(\R_+)$ such that $u(0)=0$ it holds that
$$
\int_0^\infty G\left( \frac{ |u(x)| }{x^s} \right)\,dx \leq \C_H \Phi_{s,G}(u), \qquad  
\left\| \frac{u}{x^s}\right\|_G  \leq \phi_G(\C_H) [u]_{s,G},
$$
where $\phi_G$ is given in \eqref{xxx} and $\C_H$ is the constant given in Theorem \ref{hardy}.
\end{cor}

Although a norm inequality can be  deduced from Theorem \ref{hardy}, it can also be obtained independently of the modular inequality. In fact, the following result   provides for a more accurate constant.
\begin{thm}\label{hardy.2}
Let $G$ be a Young function satisfying \eqref{cond.intro} and let  $s\in (0,1)$ be such that $sp^->1$. Then 
$$
\left\|\frac{u}{x^s} \right\|_G \leq \frac{(1+s)p^--1}{sp^--1} [u]_{s,G}
$$
for all $u\in W^{s,G}(\R_+)$.
\end{thm}

Sharpness of the Hardy constant is known both in the local and nonlocal case when $G$ is a power, see \cite{FS, Her} for instance. However, in the Orlicz setting it is  unknown even in the local case.


\section{Preliminary results} \label{sec2}
\subsection{Young functions}
A \emph{Young function} is an application $G\colon\R_+\to \R_+$ which is  continuous, nonnegative, strictly increasing, convex on $[0,\infty)$ and admits the integral formulation $G(t)=\int_0^t g(s)\,ds$. For some fixed constants $p^\pm$ we assume  that $G$ satisfies the growth condition \eqref{cond.intro}.

The \emph{complementary Young function} $\tilde G$ of a Young function $G$ is defined as
$$
\tilde G(t)=\sup\{tw-G(w)\colon w>0\}.
$$

We introduce some well-known results on Young function. See \cite{KJF,RR} for details.

\begin{lema} \label{lema.cotas}
Let $G$ be a Young function satisfying \eqref{cond.intro} and $a,t\geq 0$. Then 
\begin{align*}
  &\min\{ a^{p^-}, a^{p^+}\} G(t) \leq G(at)\leq   \max\{a^{p^-},a^{p^+}\} G(t),\tag{$G_1$}\label{G1}\\
  &G(a+t)\leq \C (G(a)+G(t)) \quad \text{with } \C:=  2^{p^+}.\tag{$G_2$}\label{G2}
 \end{align*}
\end{lema}
Condition \eqref{G2} is known as the \emph{$\Delta_2$ condition} or \emph{doubling condition}.
It can be proved that \eqref{cond.intro} implies that both $G$ and $\tilde G$ satisfy \eqref{G2}. See \cite[Theorem 3.4.4 and Theorem 3.13.9]{KJF}.

Young functions includes for instance  powers (when $g(t)=t^{p-1}$, $p^\pm=p>1$ and hence $G(t)=\tfrac{t^p}{p}$) and logarithmic perturbations of powers (when $g(t)=t\log(b+ct)$, where $p^-=1+a$, $p^+=2+a$). See \cite{KJF} for more  examples.

\subsection{Fractional Orlicz-Sobolev spaces}
Given a Young function $G$, a fractional parameter $s\in(0,1)$ and an open interval $\Omega\subseteq \R$,  we have already defined  the fractional Orlicz-Sobolev space $W^{s,G}(\Omega)$ in the introduction. We also define the following related space
$$
W^{s,G}_0(\Omega) := \{u\in W^{s,G}(\R) \colon u=0 \text{ a.e. in } \R\setminus \Omega \},
$$
which  coincides with the closure of $C^\infty_c$ functions with respect to the $\|\cdot\|_{s,G}$ norm, and it is the natural space for the well-posedness of Dirichlet problems.

 For a further generalization of these spaces we refer to   \cite{DNFBS}.

\medskip

The following result characterizes continuous functions in fractional Orlicz-Sobolev spaces.

\begin{prop} \label{continua}
Let $G$ be a Young function  satisfying \eqref{cond.intro} and let  $s\in(0,1)$  be such that $sp^->1$. Then, given an open and  bounded interval $\Omega\subset \R$, it holds that $W^{s,G}(\Omega)\subset C^{0,s-1/p^-}(\Omega)$.
\end{prop}
\begin{proof}
From \cite[Proposition 2.9 and Proposition 2.7]{FBPLS},  for any $u\in W^{s,G}(\Omega)$ we get
$$
[u]_{W_{reg}^{s,p^-}(\Omega)} + \|u\|_{L^{p^-}(\Omega)} \leq C ([u]_{W^{s,G}(\Omega)} + \|u\|_{L^{G}(\Omega)}).
$$
where
$$
[u]_{W_{reg}^{s,p^-}(\Omega)}=\left(\iint_{\Omega\times\Omega} \frac{|u(x)-u(y)|^{p^-}}{|x-y|^{n+sp^-}} \,dxdy \right)^\frac1p.
$$
Moreover, by \cite[Theorem 8.2]{hitch}, $\|u\|_{C^{0,s-1/p^-}(\Omega)} \leq C([u]_{W_{reg}^{s,p^-}(\Omega)} + \|u\|_{L^{p^-}(\Omega)})
$ 
and the result follows.
\end{proof}

\section{One-dimensional Hardy inequalities} \label{sec3}
In this section we prove our main results. First, following \cite{BK}, we link norm inequalities for linear operators with integral inequalities.

\begin{prop} \label{equiv}
Let $G$ be a Young function satisfying \eqref{cond.intro}. Suppose that the inequality $\|Tu\|_{\ve G } \leq C \|u\|_{\ve G }
$
holds for all $\ve>0$ with $C$ independent of $\ve$, where
$
\|u\|_{\ve G} =\inf \left\{ \lam\colon \int_0^\infty G\left( \frac{|u(x)|}{\lam}\right)\ve\, dx \leq 1\right\}
$
and $T$ is a linear operator. Then it holds that
$$
\int_0^\infty G(|Tu(x)|)\,dx \leq \int_0^\infty G(C|u(x)|)\,dx.
$$
\end{prop}
\begin{proof}
Given $u \in L^G(\R_+)$, define the number $\ve=\left(\int_0^\infty  G(|u|)\,dx \right)^{-1}$
and observe that $\|u\|_{\ve G  }\leq 1$. Therefore, $
\|Tu\|_{\ve G} \leq C \|u\|_{\ve G  } \leq C$, 
and then, by definition of the Luxemburg norm we get
$$
\int_0^\infty G\left( \frac{|T u|}{C} \right)\,dx \leq \frac{1}{\ve} \int_0^\infty  G\left( \frac{|T u|}{\|Tu\|_{\ve G}} \right) \ve \,dx \leq \frac{1}{\ve} = \int_0^\infty  G(|u|)\,dx.
$$
Finally, since $T$ is linear, replacing $u$  with  $Cu$ the result follows.
\end{proof}

In order to apply Proposition \ref{equiv} we  use the following  inequality due to \cite{Pal} (cf. also \cite[Corollary 4]{Mal2}).

\begin{prop} \label{palmieri}
Given $I=(0,\ell)$, $0<\ell\leq \infty$, if $\theta\in\R$ is such that $\theta<1/(p^-)'$, then, for $x^\theta u(x) \in L^G(I)$, 
$$
\left\| x^{\theta-1} \int_0^x u(t)\,dt \right\|_{L^G(I)} \leq \frac{(p^-)'}{1-\theta (p^-)' } \left\| x^\theta u(x)\right\|_{L^G(I)}.
$$
where $(p^-)'$ is the conjugated exponent of $p^-$.
\end{prop}

\begin{cor} \label{cor.palmieri}
Given $I=(0,\ell)$, $0<\ell\leq \infty$ and $s\in(0,1)$,
\begin{enumerate} 
\item[(i)] if $\theta=1-s$ and $x^{1-s} u(x) \in L^G(I)$,  for $s p^->1$ it holds that
\begin{equation} \label{palm1}
\left\| \frac{1}{x^s} \int_0^x u(t)\,dt \right\|_{L^G(I)} \leq \frac{p^-}{sp^--1} \left\| x^{1-s} u(x)\right\|_{L^G(I)},
\end{equation}
\item[(ii)] if $\theta=-s$ and $x^{-s} u(x) \in L^G(I)$,  for $(1+s) p^->1$ it holds that
\begin{equation} \label{palm2}
\left\| \frac{1}{x^{1+s}} \int_0^x u(t)\,dt \right\|_{L^G(I)} \leq \frac{p^-}{(1+s)p^--1} \left\| x^{-s} u(x)\right\|_{L^G(I)}.
\end{equation}
\end{enumerate}
\end{cor}

We prove now a key lemma for our arguments.
\begin{lema} \label{lema.loca}
Let $G$ be a Young function satisfying \eqref{cond.intro} and let $s\in(0,1)$ be such that $sp^->1$. Given  $u$ such that $x^{-s} u(x) \in L^G(\R_+)$, then we have that
\begin{equation} \label{hardy.orlicz.local2}
\int_0^\infty G\left(\left| \frac{1}{x^{s}}\int_0^x \frac{u(t)}{t}\,dt \right|\right)\,dx \leq \c_H \int_0^\infty G\left(   \frac{|u(x)|}{x^{s}}\right)\,dx.
\end{equation}
The constant $\c_H$ is given by $\c_H=\psi_G\left( \frac{p^-}{sp^--1}\right)$, where $\psi_G$ is given in \eqref{xxx}.
\end{lema}

\begin{proof}
Given  $G$ satisfying \eqref{cond.intro}, we define the Young function $G_\ve:=\ve G$ for some $\ve>0$. It is immediate  that $G_\ve$ satisfies \eqref{cond.intro} with the same constants that $G$. 

Let $u$ be a fixed function such that $x^{-s} u(x) \in L^G(\R_+)$, then also $x^{-s} u(x) \in L^{\ve G}(\R_+)$. Hence, by applying Corollary \ref{cor.palmieri}  item \eqref{palm1} to  $x^{-1}u(x)$ we find that
\begin{equation} \label{ec.norma}
\left\|  \frac{1}{x^{s}}\int_0^x \frac{u(t)}{t}\,dt \right\|_{G_\ve} \leq \frac{p^-}{sp^--1 }  \left\| \frac{u}{x^{s}} \right\|_{G_\ve} \qquad \forall \ve>0.
\end{equation}
Since \eqref{ec.norma}  holds with constant  independent of $\ve$, from Proposition \ref{equiv} we get that
$$
\int_0^\infty G\left( \left| \frac{1}{x^{s}}\int_0^x \frac{u(t)}{t}\,dt \right|\right)\,dx \leq \int_0^\infty G\left(\frac{p^-}{sp^--1 }   \frac{  |u(x)|}{x^{s}}\right)\,dx
$$
and the result follows by using  condition \eqref{G1}.
\end{proof}

 We are ready to prove our   main result.
\begin{proof}[Proof of Theorem \ref{hardy}]
We assume without loss of generality that 
$$
\lim_{x\to 0} \frac{1}{x}\int_0^x u(x)\,dx =0.
$$
Then, the general case will follow by applying Theorem \ref{hardy} to $u-u_0$.

Given $u\in W^{s,G}(\R_+)$ we consider the auxiliary function
$$
 v(x)= u(x)-\frac{1}{x}\int_0^x u(t)\,dt.
$$
For $0<a<b<\infty$, using integration by parts  it is straightforward to see that
\begin{align*}  
\begin{split}
	\int_{a}^{b} \frac{v(x)}{x}\,dx  &= \int_a^b \frac{u(x)}{x}\,dx-\int_a^b\frac{1}{x^2}\int_0^x u(s)\,ds\,dx
	\\
	 &=\int_a^b \frac{u(x)}{x}\,dx + \frac{1}{x}\int_0^x u(s)\,ds \Big|_a^b - \int_a^b \frac{u(x)}{x}\,dx\\
	&= \frac{1}{b}\int_0^b u(x)\,dx -  \frac{1}{a}\int_0^a u(x)\,dx.
\end{split}
\end{align*}
Taking $b=t$ and $a\to 0$ in the last expression we get
$
\int_{0}^{t} \frac{v(x)}{x}\,dx  =  \frac{1}{t}\int_0^t u(x)\,dx,
$
and
\begin{equation} \label{rel.uv}
u(x)=v(x)+\int_0^x \frac{v(t)}{t}\,dt.
\end{equation}
We prove now that $x^{-s} v(x)\in L^G(\R_+)$. Indeed, by definition of $v$ we have that
\begin{align*}
\int_0^\infty G\left( \frac{|v(x)|}{x^s}\right)    \,dx = \int_0^\infty G\left( \frac{  |u(x)-\frac{1}{x}\int_0^x u(t)\,dt| }{x^s} \right) \,dx.
\end{align*}
Now, by using Jensen's inequality we get
\begin{align*} 
\int_0^\infty G\left( \frac{ |u(x)-\frac{1}{x}\int_0^x u(t)\,dt| }{x^s} \right) \,dx &\leq 
\int_0^\infty G\left(  \frac{1}{x} \int_0^x \frac{|u(x)-u(t)|}{x^s} dt \right) \,dx\\
&\leq \int_0^\infty \frac{1}{x} \int_0^x G\left(    \frac{|u(x)-u(t)|}{x^s}  \right) dt \,dx,
\end{align*}
and the last term in the inequality above can be bounded as
\begin{align*}
\int_0^\infty \frac{1}{x} \int_0^x G\left(    \frac{|u(x)-u(t)|}{x^s}  \right) dt \,dx &\leq \int_0^\infty  \int_0^x G\left(    \frac{|u(x)-u(y)|}{|x-y|^s}  \right) \frac{dy\,dx}{|x-y|}\\
&\leq  \int_0^\infty  \int_0^\infty G\left(    \frac{|u(x)-u(y)|}{|x-y|^s}  \right) \frac{dy\,dx}{|x-y|}.
\end{align*}

Since $x^{-s} v(x)\in L^G(\R_+)$,  we are in position to apply Lemma \ref{lema.loca} to $v$, obtaining 
\begin{equation} \label{ec.interm}
\int_0^\infty G \left(\left| \frac{1}{x^s}   \int_0^x \frac{v(t)}{t}\,dt   \right| \right)  \,dx \leq \c_H \int_0^\infty G\left( \frac{ |v(x)| }{x^s} \right) \,dx.
\end{equation}

Therefore, from \eqref{rel.uv}, property \eqref{G2} and \eqref{ec.interm} we find that
\begin{align*}
\int_0^\infty G\left( \frac{ |u(x)| }{x^s} \right)\,dx &\leq  
\C\int_0^\infty G\left( \frac{ |v(x)| }{x^s} \right) \,dx +\C\int_0^\infty G \left(\frac{1}{x^s}   \left|\int_0^x \frac{v(t)}{t}\,dt  \right|  \right)  \,dx \\
&\leq \C(1+\c_H)\int_0^\infty G\left( \frac{ |v(x)| }{x^s} \right) \,dx\\
&\leq  \C(1+\c_H)\Phi_{s,G}(u)
\end{align*}
giving the desired inequality.
\end{proof}

As a corollary, we obtain the Hardy inequality for norms stated in  Corollary \ref{hardy.cor}.

\begin{proof}[Proof of Corollary \ref{hardy.cor}]
We assume without loss of generality that 
$$
\lim_{x\to 0} \frac{1}{x}\int_0^x u(x)\,dx =0.
$$
Then, the general case will follow by applying the inequality to $u-u_0$.

By using  Theorem \ref{hardy} together with \eqref{G1} one gets that $
\Phi_G(\tfrac{u}{x^s})\leq \Phi_{s,G}(\tilde \C_H u)$ for $u\in W^{s,G}(\R_+)$, where $\tilde \C_H = \max\{\C_H^\frac{1}{p^-}, \C_H^\frac{1}{p^+}\}$. Then, given $u\in W^{s,G}(\R_+)$, from the last expression  we  find that
$
\Phi_G\left(u/\tilde\C_H [u]_{s,G}  x^s\right) \leq \Phi_{s,G}\left(u/[u]_{s,G}\right) \leq 1.
$
Hence, by definition of the Luxemburg norm we get
$$
\left\| \frac{u}{x^s}\right\|_G=\inf\left\{ \lam : \Phi_G\left(\frac{u}{\lam x^s }\right)\leq 1 \right\} \leq \tilde \C_H [u]_{s,G}
$$
and the proof concludes.
\end{proof}

Now, we provide for the proof of the norm inequality given in Theorem \ref{hardy.2}.

\begin{proof}[Proof of Theorem \ref{hardy.2}]
Given $u\in W^{s,G}(\R_+)$, by using the triangular inequality for the Luxemburg norm we obtain  that
\begin{align*}
\left\|\frac{u}{x^s} \right\|_G &\leq \left\| \frac{u-\tfrac{1}{x}\int_0^x u(y)\,dy}{  x^s}\right\|_G + \left\|  \frac{1}{x^{1+s}}\int_0^x u(y)\,dy \right\|_G:= (i) + (ii).
\end{align*}
Let us find a bound for $(i)$. By using Jensen's inequality we get
\begin{align*} 
\int_0^\infty G\Big( &\frac{ \left|u(x)-\frac{1}{x}\int_0^x u(t)\,dt \right| }{x^s} \Big) \,dx \leq 
\int_0^\infty G\left(  \frac{1}{x} \int_0^x \frac{|u(x)-u(t)|}{x^s} dt \right) \,dx\\
&\leq \int_0^\infty \frac{1}{x} \int_0^x G\left(    \frac{|u(x)-u(t)|}{x^s}  \right) dt \,dx\leq \int_0^\infty  \int_0^x G\left(    \frac{|u(x)-u(y)|}{|x-y|^s}  \right) \frac{dy\,dx}{|x-y|}\\&\leq \Phi_{s,G}(u).
\end{align*}
The inequality above applied to $u/[u]_{s,G}$ gives that
$$
\int_0^\infty G\left( \frac{\left| u(x)-\frac{1}{x}\int_0^x u(y)\,dy\right|}{x^s [u]_{s,G}} \right) \,dx \leq \Phi_{s,G}\left(\frac{u}{[u]_{s,G}}\right)\leq 1,
$$
from where, from the definition of the Luxemburg norm, we obtain that $
(i)$ is bounded by $[u]_{s,G}$. Expression \eqref{palm2} from Corollary \ref{cor.palmieri}  gives  that $(ii)$ is less than $\frac{p^-}{(1+s) p^--1 } \left\| \frac{u}{x^s} \right\|_G$. From these two relations we get that
$$
\left\|\frac{u}{x^s} \right\|_G \leq 
 [u]_{s,G}+ \frac{p^-}{(1+s)p^--1} \left\| \frac{u}{x^s} \right\|_G.
$$
Finally, the condition $sp^->1$ leads to
$$
\C_H=\left(1- \frac{p^-}{(1+s) p^--1 } \right)^{-1}=\frac{(1+s)p^--1}{sp^--1}>0
$$
and the desired inequality is obtained.
\end{proof}

\section{Applications and examples} \label{sec5}
\subsection{Lower bound of Dirichlet eigenvalues}  
Given a Young function $G$ satisfying \eqref{cond.intro} and a fractional parameter $s\in(0,1)$ such that $sp^->1$, consider the    eigenvalue problem for the fractional $g-$Laplacian operator in an open and bounded interval $\Omega\subset \R^+$ (see \cite{FBS,S}):
\begin{align} \label{ec.peso.2}
\begin{cases}
(-\Delta_g)^s u= \lam g(u) u &\quad \text{ in } \Omega,\\
u=0 &\quad \text{ on } \R  \setminus \Omega.
\end{cases}
\end{align}
This operator  is well defined between $W^{s,G}(\R)$ and its dual, and acts as
$$
(-\Delta_g)^s u(x):=2p.v. \int_{-\infty}^\infty g(|D_su|)\frac{D_s u}{|D_s u|} \frac{dy}{|x-y|^{1+s}},
$$
where $D_s u(x,y)=\frac{u(x)-u(y)}{|x-y|^s}$. The following representation formula holds
$$
\langle (-\Delta_g)^s u,v\rangle= \int_{-\infty}^\infty\int_{-\infty}^\infty g(|D_s u|)\frac{D_s u}{|D_s u|} D_s v\frac{dxdy}{|x-y|} \qquad \forall v\in W^{s,G}(\R).
$$
The natural space for solutions of \eqref{ec.peso.2} is  $W^{s,G}_0(\Omega)$, i.e.,  functions in $W^{s,G}(\R)$ such that $u=0$ in $\R\setminus \Omega$.  

Due to the possible lack of homogeneity of the problem, eigenpairs depend strongly on the normalization: for each $\alpha>0$,   $u_\alpha \in \mathcal{X}_\alpha :=\left\{ W^{s,G}_0(\Omega) \colon \Phi_G(u)=\alpha \right\}$ is an eigenfunction of \eqref{ec.peso.2} with eigenvalue $\lam_\alpha$ if it holds that
\begin{equation*}  
\langle (-\Delta_g)^s u_\alpha,v \rangle = \lam_\alpha \int_\Omega   g\left(u_\alpha \right)u_\alpha\frac{u_{\alpha}}{|u_{\alpha}|}v\,dx \quad \forall v\in W^{s,G}_0(\Omega).
\end{equation*}
On the other hand, for each $\alpha>0$ we define the minimizer  (see \cite[Theorem 2.12]{FBPLS}) 
\begin{equation} \label{alpha.2}
\Lambda_\alpha:=\inf \{   \Phi_{s,G}(u)/\Phi_G(u) \colon u\in \mathcal{X}_\alpha \}.
\end{equation}

In  \cite{S} it is proved that for each $\alpha>0$ the infimum in \eqref{alpha.2} is attained for  some  function $u_\alpha  \in \mathcal{X}_\alpha$, and then, by a Lagrange's multiplier argument it is deduced the existence of an eigenvalue $\lam_\alpha$  of \eqref{ec.peso.2} with  eigenfunction  $u_\alpha$. However, in contrast with the case of powers,  in general  $\lam_\alpha$ does not admit a variational characterization and $\lam_\alpha \neq \Lambda_\alpha$. Both constants are comparable each other. Indeed, 
$
\frac{p^-}{p^+}\Lambda_\alpha\leq \lam_\alpha \leq \frac{p^+}{p^-}\Lambda_\alpha.
$

Using Lemma \ref{lema.cotas} it is easy to see that for any function $u\in \mathcal{X}_\alpha$ it holds that
$$
\int_{\Omega} G(|u|)\,dx \leq \int_\Omega G\left(\frac{|u(x)|}{x^s} \text{diam}(\Omega) \right)\,dx \leq \psi_G(\text{diam}(\Omega)) \int_\Omega G\left( \frac{|u(x)|}{x^s} \right)\,dx,
$$
where $\psi_G$ is given in \eqref{xxx}. Hence, from Theorems \ref{hardy}  we get that
$$
(\psi_G(\text{diam}(\Omega))\C_H)^{-1}\leq \Lambda_\alpha \leq \frac{p^+}{p^-} \lam_\alpha.
$$

\subsection{Lower bound of weighted eigenvalues}  
With the same ideas of \cite{S}, we can consider eigenvalues $\lam_\alpha$ and minimizers $\Lambda_\alpha$  of the weighted problem
\begin{align} \label{ec.peso}
\begin{cases}
(-\Delta_g)^s u= \tfrac{\lambda}{|x|^s} g\big(\tfrac{|u|}{x^s}\big) \frac{u}{|u|} &\quad \text{ in } \Omega,\\
u=0 &\quad \text{ on } \R  \setminus \Omega.
\end{cases}
\end{align}
where $\Omega\subset \R^+$ is an open and bounded interval. Given $\alpha>0$, a number $\lam_\alpha$ is an eigenvalue of \eqref{ec.peso} with eigenfunction   $u_\alpha\in \mathcal{X}_\alpha :=\left\{ W^{s,G}_0(\Omega) \colon \int_\Omega G\left(\frac{|u(x)|}{x^s}\right)\,dx=\alpha \right\}$ if it holds that
\begin{equation*}  
\langle (-\Delta_g)^s u_\alpha,v \rangle = \lam_\alpha \int_\Omega   g\left(\frac{|u_\alpha(x)|}{|x|^s}\right)\frac{u_\alpha(x)}{|u_\alpha(x)|}\frac{v(x)}{x^s}\,dx \quad \forall v\in W^{s,G}_0(\Omega).
\end{equation*}
We also define the number $\Lambda_\alpha:=\inf \left\{ \Phi_{s,G}(u)/\int_\Omega G\left(\frac{|u(x)|}{x^s}\right)\,dx \colon u\in \mathcal{X}_\alpha\right\}$.

We  have again in this case that $\frac{p^-}{p^+}\Lambda_\alpha\leq \lam_\alpha \leq \frac{p^+}{p^-}\Lambda_\alpha$. As a direct consequence of Theorem \ref{hardy}  we get   $\frac{1}{\C_H} \leq \Lambda_\alpha \leq \frac{p^+}{p^-} \lam_\alpha$.

\section*{Acknowledgements}
This paper is partially supported by grants UBACyT 20020130100283BA, CONICET PIP 11220150100032CO and ANPCyT PICT 2012-0153. 

\end{document}